\documentclass[11pt]{article} 
\usepackage{multicol} 
\usepackage{newlfont}
\usepackage{amsmath,amssymb}

\def\rit{\mathbb{R}}
\def\zit{\mathbb{Z}}

\def\ppit{\mathbb{P}} 
 
\def\cit{\mathbb{C}} 
\def\fit{\mathbb{F}}

\newcommand{\pf}{{\em Proof.~}}
\newcommand{\qed}{\hfill~~\mbox{$\Box$}}

\newenvironment{proof}{\smallskip \noindent \pf}{\qed \bigskip}

\newtheorem{theorem}{Theorem}[subsection]
\newtheorem{proposition}[theorem]{Proposition}

\newtheorem{lemma}[theorem]{Lemma}
\newtheorem{corollary}[theorem]{Corollary}

\newtheorem{remark}[theorem]{Remark}
\newtheorem{example}[theorem]{Example}

\begin{document}

\title{\bf Examples of limits of Frobenius (type) structures: the singularity case}                                                      
\author{\sc Antoine Douai
\thanks{A.Douai: Laboratoire J.A Dieudonn\'e, UMR CNRS 6621,
Universit\'e de Nice, Parc Valrose, F-06108 Nice Cedex 2, France. Email: douai@unice.fr . Partially supported by ANR grant BLAN08-1 309225 (SEDIGA)
}}

\maketitle

\begin{abstract}
We give examples of families of Frobenius type structures on the punctured plane and we study their limits at the boundary. We then discuss the existence of a limit Frobenius manifold. We also give an example of a logarithmic Frobenius manifold. 
\end{abstract}

\begin{center} PLEASE NOTE THAT THIS PREPRINT HAS BEEN SUPERSEDED BY arXiv:0909.4063
 
\end{center}

\section{Introduction}
Let $w_{1},\cdots , w_{n}$ be positive integers and
$f:(\cit^{*})^{n}\rightarrow\cit$ be the Laurent polynomial defined by 
$$f(u_{1},\cdots ,u_{n})=u_{1}+\cdots +u_{n}+\frac{1}{u_{1}^{w_{1}}\cdots u_{n}^{w_{n}}}.$$
It has been explained in \cite{DoSa2} how to attach to $f$ a canonical Frobenius manifold: 
the two main ingredients are a Frobenius type structure on a point, that is a tuple
$$(E^{o}, R_{0}^{o}, R_{\infty}^{o}, g^{o})$$
where $E^{o}$ is a finite dimensional vector space over $\cit$,  $g^{o}$ is a symmetric and nondegenerate bilinear form on $E^{o}$, $R_{0}^{o}$ and $R_{\infty}^{o}$ 
are two endomorphisms of $E^{o}$ such that $R_{\infty}^{o}+(R_{\infty}^{o})^{*}=nId$ and $(R_{0}^{o})^{*}=R_{0}^{o}$ ($^{*}$ denotes the adjoint with respect to $g^{o}$) and a pre-primitive and homogeneous section of $E^{o}$, namely a section which is a cyclic vector of $R_{0}^{o}$ and an eigenvector of  $R_{\infty}^{o}$. 
The canonical solution of the Birkhoff problem
for the Brieskorn lattice of $f$ given by M. Saito's method yields the required canonical Frobenius type structure. 
This is the punctual construction. This gives, for $w_{1}=\cdots =w_{n}=1$, the mirror partner of the projective space $\ppit^{n}$ (see \cite{Ba}), and more generally the mirror partner of the weighted projective space $\ppit (1,w_{1},\cdots ,w_{n})$ (see \cite{Man} and \cite{Coa}).

 The purpose of these notes is to give analogous results for the deformation
$F:(\cit^{*})^{n}\times X\rightarrow \cit$ of $f$ defined by  
$$F(u_{1},\cdots ,u_{n},x)=u_{1}+\cdots +u_{n}+\frac{x}{u_{1}^{w_{1}}\cdots u_{n}^{w^{n}}}$$
where $X:=\cit^{*}$ and then to discuss the existence of a ''limit'' Frobenius manifold as $x$ approaches $0$. This kind of problem is also considered in \cite{dGMS}, using another strategy (we will not use the reference \cite{DoSa1} at all) and for a different class of functions. Notice however that the case $w_{1}=\cdots =w_{n}=1$ is common to both papers.

Let us precise the situation: let
$$G=\frac{\Omega^{n}(U)[\tau ,\tau^{-1}, x, x^{-1}]}{(d_{u}-\tau d_{u}F)\wedge\Omega^{n-1}(U)[\tau ,\tau^{-1}, x ,x^{-1}]}$$
be the (Fourier-Laplace transform of the) Gauss-Manin system of $F$
and
$$G_{0}=\frac{\Omega^{n}(U)[\tau^{-1}, x, x^{-1}]}{(\tau^{-1}d_{u}-d_{u}F)\wedge \Omega^{n-1}(U)[\tau^{-1}, x ,x^{-1}]}$$
be (the Fourier-Laplace transform of) its Brieskorn lattice, where the notation $d_{u}$ means that the differential is taken with respect to $u$ only.
$G$ is equipped with a connection $\nabla$ defined by   
$$\nabla_{\partial_{\tau}}(\omega_{i}\tau^{i})=i\omega_{i}\tau^{i-1}-F\omega_{i}\tau^{i}$$
and
$$\nabla_{\partial_{x}}(\omega_{i}\tau^{i})=\frac{\partial \omega_{i}}{\partial x}\tau^{i}-\frac{\partial F}{\partial x}\omega_{i}\tau^{i+1}.$$
In particular, if we put $\theta :=\tau^{-1}$,
 $G_{0}$ is stable under the action of $\theta^{2}\nabla_{\partial_{\theta}}$. One defines in the same way the Gauss-Manin system ({\em resp.} the Brieskorn lattice) $G^{o}$ ({\em resp.} $G_{0}^{o}$) of the Laurent polynomial $f$ (see \cite[section 4]{Do1} for details). It turns out 
that one can solve the Birkhoff problem for $G_{0}$ on the whole $X$: $G_{0}$ is a free $\cit [\theta , x, x^{-1}]$-module of rank $\mu =1+w_{1}+\cdots +w_{n}$ and there exists a basis of  $G_{0}$ 
in which the matrix of the connection $\nabla$ takes the form
$$(\frac{A_{0}(x)}{\theta}+A_{\infty})\frac{d\theta}{\theta}+(-\frac{A_{0}(x)}{\mu\theta}+R)\frac{dx}{x},$$
$A_{0}(x)$ being a $\mu\times\mu$ matrix with coefficients in $\cit [x]$, $A_{\infty}$ and $R$ being diagonal matrices with constant coefficients (see proposition \ref{libre}).
This gives a Frobenius type structure on $X$ (see \cite{Do1} and \cite{HeMa}), that is a tuple 
$$(X, E, R_{0}, R_{\infty}, \Phi, \bigtriangledown)$$
where the different objects involved satisfy some natural compatibility relations
which can be extended, and this is done in section \ref{dualite}, to a Frobenius type structure with metric (corollary \ref{STF})
$$\fit =(X, E, R_{0}, R_{\infty}, \Phi, \bigtriangledown ,g)$$
which will be the central object of these notes. It should be emphasized that the metric $g$ plays here a fundamental role. This Frobenius type structure, together with the data of a pre-primitive, homogeneous and $\bigtriangledown$-flat form, yields also a Frobenius manifold on $\Delta\times (\cit^{\mu -1},0)$ where $\Delta$ denotes the open disc of radius one, centered at $x=1$ (see \cite{Do1}, \cite{HeMa}): we will use it first to compare the canonical Frobenius manifolds attached to the different polynomials $F_{x}:=F(.\ ,x)$, $x\in \Delta$, by the punctual construction (see section \ref{construction}).

The second part of these notes (section \ref{limite}) is devoted to the study of the limit, as $x$ approaches $0$, of the Frobenius type structure $\fit$. This limit is defined using Deligne's canonical extension ${\cal L}^{\varphi}$ such that the eigenvalues of the residue of $\nabla_{\partial_{x}}$ are contained in $[0,1[$: 
this lattice is easily described in our situation. The key point is that $gr^{V}({\cal L}^{\varphi}/x{\cal L}^{\varphi})$, 
the graded module associated with the Malgrange-Kashiwara $V$-filtration at $x=0$, yields a Frobenius type structure on a point which can thus be seen as the canonical limit Frobenius type structure (notice that this result is not always true if we consider ${\cal L}^{\varphi}/x{\cal L}^{\varphi}$ instead of $gr^{V}({\cal L}^{\varphi}/x{\cal L}^{\varphi})$, that is if we forget the graduation). In order to define a canonical limit Frobenius manifold, we still need a pre-primitive and homogeneous section 
of this limit Frobenius type structure, see again \cite{HeMa} and the references therein: we show in section \ref{limite} that such a section exists if and only if $w_{1}=\cdots =w_{n}=1$. In this case, we give in section \ref{varfrob} an explicit description of this canonical limit Frobenius manifold. In general, that is  
if there is an $w_{i}$ such that $w_{i}\geq 2$, the situation is less clear for the following reasons: first, we do not have a general statement saying that one can derive a Frobenius manifold from the canonical limit Frobenius type structure (nevertheless, it should emphasized that we do not assert here that such a limit does not exist); second, even if it happens to be the case, one could get several Frobenius manifolds which can be difficult to compare. 

The last section is devoted to logarithmic Frobenius manifolds:
if $w_{1}=\cdots =w_{n}=1$, we show how to get, with the help now of a suitable extension of $G_{0}$ at $x=0$, a Frobenius type structure with logarithmic pole along $\{x=0\}$ in the sense of \cite[Definition 1.6]{R}, yielding a logarithmic Frobenius manifold. If there exists a weight $w_{i}$ such that $w_{i}\geq 2$, we have all the tools to define a Frobenius type structure with a logarithmic pole along $\{x=0\}$, except the metric: the symmetric bilinear form constructed here is flat but not non-degenerate.

The starting point of this paper is the reference \cite{B} in which A. Bolibruch discusses the properties of the limit of an isomonodromic family of Fuchsian systems. It happens that, 
in our geometric situation, this family is produced, {\em via} an inverse Fourier-Laplace transformation, by a solution of the Birkhoff problem for the Brieskorn lattice of a rescaling $H(u,x)=xf(u)$ of a tame regular morse function $f$, see \cite[Chapitre VI]{Sab1}. This leads naturally to the following question: given a Frobenius type structure on $X$, what can we expect at the limit? Finally, the reference \cite{dGMS}, and I thank C. Sabbah for a discussion about this, explains the choice of the deformation $F$ made in these notes (the computations for the rescaling $H$ are similar and easier to the ones performed here). 
The last section grew up after a discussion with C. Sevenheck who suggested me to work with the natural extensions ${\cal L}_{0}^{\varphi}$ of $G_{0}$: I thank him for that.

\section{The canonical solution of the Birkhoff problem for $G_{0}^{o}$}
\label{absolu}

\subsection{Preliminaries}  

Let ${\cal S}_{w}$ be the {\em disjoint} union of the sets
$$\{\frac{\ell\mu}{w_{i}}|\, \ell =0,\cdots ,w_{i}-1\}$$
for $i=0,\cdots,n$, where we put $w_{0}=1$. Its cardinal is equal to $\mu :=1+w_{1}+\cdots +w_{n}$. We number the elements of ${\cal S}_{w}$ from $0$ to $\mu -1$ in an increasing way and write
$${\cal S}_{w}=\{s_{0},s_{1},\cdots ,s_{\mu -1}\}$$
(thus, $s_{k}\leq s_{k+1}$). 
Notice that $0\leq\frac{s_{k}}{\mu}<1$.
Define, for $k=0,\cdots ,\mu -1$, 
$$\alpha_{k}=k-s_{k}.$$

\begin{lemma} \label{lessk} One has $s_{0}=\cdots =s_{n}=0$, $s_{n+1}=\frac{\mu}{max_{i}w_{i}}$
and
$s_{k}+s_{\mu +n-k}=\mu$
for $k\geq n+1$.
\end{lemma}
\begin{proof}
See \cite[p. 2]{DoSa2}. 
\end{proof}
\begin{corollary} \label{lesalphak}
One has $\alpha_{0}=0, \cdots , \alpha_{n}=n$, $\alpha_{k+1}\leq \alpha_{k}+1$ for all $k$,
$$\alpha_{k}+\alpha_{\mu +n-k}=n$$
for $k\geq n+1$ and
$$\alpha_{k}+\alpha_{n-k}=n$$
for $k=0,\cdots ,n$.
\end{corollary}
\begin{proof}
One has $\alpha_{k+1}\leq \alpha_{k}+1$ because $(s_{k})$ is increasing. The remaining assertions are clear.
\end{proof}

\subsection{The Birkhoff problem for $G_{0}^{o}$}

Let $A_{0}^{o}$ and $A_{\infty}$ be the $\mu\times\mu$ matrices defined by                                                    
$$A_{\infty}=diag (\alpha_{0},\cdots ,\alpha_{\mu -1})$$
and
$$A_{0}^{o}=\left ( \begin{array}{cccccc}
0   & 0   & 0 & \cdots & 0   & \mu\\
\mu   & 0   & 0 & \cdots & 0   & 0\\
0   & \mu  & 0 & \cdots & 0   & 0\\
..  & ... & . & \cdots & .   & .\\
..  & ... & . & \cdots & .   & .\\
0   & 0   & . & \cdots & \mu   & 0 
\end{array}
\right ).$$

\begin{example}\label{example} We will work with the following examples:\\ 
(1) $n=2$ and $w_{1}=w_{2}=2$: one has $\mu =5$ and $A_{\infty}=diag (0,1,2,\frac{1}{2},\frac{3}{2}).$\\
(2) $w_{1}=\cdots =w_{n}=1$: one has $\mu =n+1$ and     
 $A_{\infty}=diag (0,1,\cdots ,n)$. 
\end{example}

\noindent The following results are shown in \cite{DoSa2}:
 
\begin{lemma} \label{start}
(1) $G_{0}^{o}$ is a free $\cit [\theta]$-module of rank $\mu$, equipped with a connection $\nabla$ with a pole of Poincar\'e rank less or equal to $1$ at $\theta =0$.\\
(2) The Birkhoff problem for $G_{0}^{o}$ has a solution: there exists a basis $\omega^{o} =(\omega_{0}^{o},\cdots ,\omega_{\mu -1}^{o})$ of $G_{0}^{o}$ over $\cit [\theta]$ in which the matrix of the connection $\nabla$ takes the form
$$(\frac{A_{0}^{o}}{\theta}+A_{\infty})\frac{d\theta}{\theta}.$$
Moreover, the eigenvalues of $A_{\infty}$ run through the spectrum at infinity of the polynomial $f$.
\end{lemma} 

\noindent We even have a little bit more: the basis $\omega^{o}$ constructed in {\em loc. cit.} is the {\em canonical} solution of the Birkhoff problem given by M. Saito's method. In particular, it is compatible with the $V$-filtration at $\tau =0$ (see the last assertion of \cite[Proposition 3.2]{DoSa2}). One then has         
$A_{0}^{o}V_{\alpha}\subset V_{\alpha +1}$: in other words, if $(A_{0})_{ij}\neq 0$ then $\alpha_{i-1}\leq\alpha_{j-1}+1$. One can moreover endow $G_{0}^{o}$ with a ''metric'': this is discussed in section \ref{dualite}.\\

\section{The Birkhoff problem for $G_{0}$}
\label{deformation}
We give in this section the counterpart of the previous results for $G_{0}$. 
We will put $u_{0}:=\frac{1}{u_{1}^{w_{1}}\cdots u_{n}^{w_{n}}}$ and $\omega_{0}:=\frac{du_{1}}{u_{1}}\wedge\cdots\wedge\frac{du_{n}}{u_{n}}$

\subsection{A natural solution}
\label{solnat}

Define
$$\Gamma_{0}=\{(y_{1},\cdots ,y_{n})\in\rit^{n} |  y_{1}+\cdots +y_{n}=1\},$$
$$\Gamma_{j}= \{(y_{1},\cdots ,y_{n})\in\rit^{n} |  y_{1}+\cdots +y_{j-1}+(1-\frac{\mu}{w_{j}})y_{j}+\cdots +y_{n}=1\}$$
for $j=1,\cdots ,n$,
$$\chi_{\Gamma_{0}}=u_{1}\frac{\partial}{\partial u_{1}}+\cdots +u_{n}\frac{\partial}{\partial u_{n}}$$
and, for $j=1,\cdots ,n$,
$$\chi_{\Gamma_{j}}= u_{1}\frac{\partial}{\partial u_{1}}+\cdots +u_{j-1}\frac{\partial}{\partial u_{j-1}}+(1-\frac{\mu}{w_{j}})u_{j}\frac{\partial}{\partial u_{j}}+\cdots +u_{n}\frac{\partial}{\partial u_{n}}.$$
Define also, for $g=u_{1}^{a_{1}}\cdots u_{n}^{a_{n}}$,
$$\phi_{\Gamma_{j}}(g)=a_{1}\cdots +a_{j-1}+(1-\frac{\mu}{w_{j}})a_{j}+\cdots +a_{n}$$
and
$$h_{\Gamma_{j}}=\chi_{\Gamma_{j}}(F)-F.$$
We thus have $h_{\Gamma_{0}}=-\mu x u_{0}$ and $h_{\Gamma_{j}}=-\frac{\mu}{w_{j}} u_{j}$ if $j=1,\cdots ,n$. 

\begin{lemma} \label{lemmebase} One has, for any monomial $g$, the equality
 $$(\tau\partial_{\tau}+\phi_{\Gamma_{j}}(g))g\omega_{0}=\tau h_{\Gamma_{j}}g\omega_{0}$$
in $G$. In particular, one has
$$\tau\partial_{\tau}\omega_{0}=\tau h_{\Gamma_{0}}\omega_{0}.$$
\end{lemma}
\begin{proof}
 Direct computation.
\end{proof}

\noindent This lemma is the starting point in order to solve the Birkhoff problem for $G_{0}$, as it has been the starting point to solve the one for $G_{0}^{o}$ (see \cite[section 3]{DoSa2}).
Put $\omega_{1}=u_{0}\omega_{0}$: the equality  
$$\tau\partial_{\tau}\omega_{0}=\tau h_{\Gamma_{0}}\omega_{0}$$
becomes
$$-\frac{1}{\mu}\tau\partial_{\tau}\omega_{0}=x\tau\omega_{1}.$$
Iterating the process (the idea is to define $\omega_{2}=-\frac{1}{\mu}\omega_{1}h_{\Gamma_{1}}$ {\em etc}...), one gets sections $\omega_{1},\cdots ,\omega_{\mu -1}$ of $G$ satisfying
$$-\frac{1}{\mu}(\tau\partial_{\tau}+\alpha_{k})\omega_{k}=\tau\omega_{k+1}$$
for $k=1,\cdots ,\mu -1$ (we put $\omega_{\mu}=\omega_{0}$): this can be done as \cite[section 2 and proof of proposition 3.2]{DoSa2}. We thus have
$$\omega_{k}=u^{a(k)}\omega_{0}$$
where the multi-indices $a(k)$ are defined in \cite[p. 3]{DoSa2}. 

Define,
for $x\in X$, 
$$A_{0}(x):=\left ( \begin{array}{cccccc}
0   & 0   & 0 & \cdots & 0   & \mu\\
\mu x  & 0   & 0 & \cdots & 0   & 0\\
0   & \mu  & 0 & \cdots & 0   & 0\\
..  & ... & . & \cdots & .   & .\\
..  & ... & . & \cdots & .   & .\\
0   & 0   & . & \cdots & \mu   & 0 
\end{array}
\right )$$
and
 $$R=diag (0,-1,\cdots ,-1, -s_{\mu -1}/\mu,\cdots , -s_{n+1}/\mu)$$
(the entry $-1$ is counted $n$ times).

\begin{example} (1) If $n=2$ and $w_{1}=w_{2}=2$, one has  
$R=-diag (0,1,1,\frac{1}{2},\frac{1}{2}).$\\
(2) If $w_{1}=\cdots =w_{n}=1$, one has
$R=- diag (0,1,1,\cdots ,1).$
\end{example}

\begin{proposition}\label{libre} 
(1) $G_{0}$ is a free $\cit [x,x^{-1},\theta ]$-module of rank $\mu$ and $\omega =(\omega_{0},\cdots ,\omega_{\mu -1})$ is a basis of it.\\
(2) In the basis $\omega$, the matrix of the connection $\nabla$ takes the form
$$(\frac{A_{0}(x)}{\theta}+A_{\infty})\frac{d\theta}{\theta}+(R-\frac{A_{0}(x)}{\mu\theta})\frac{dx}{x}$$
where $A_{\infty}$ is the diagonal matrix defined in section \ref{absolu}.
\end{proposition}
\begin{proof} (1) 
One shows that $G_{0}$ is finitely generated as in \cite[proposition 3.2]{DoSa2}, with the help of lemma \ref{lemmebase}. 
To show that it is free notice that, again by \cite[proposition 3.2]{DoSa2}, a section of the kernel of the surjective map
$$(\cit [x,x^{-1}])^{\mu}\rightarrow G_{0}\rightarrow 0$$
is given by $\mu$ Laurent polynomials which vanishes everywhere.
Let us show (2): the assertion about $\nabla_{\partial_{\theta}}$ is clear 
(by definition of the $\omega_{k}$'s). Recall that the action of $\nabla_{\partial_{x}}$ is defined, for $\eta\in G_{0}$, by 
 $$\nabla_{\partial_{x}}(\eta)=-\frac{\partial F}{\partial x}\eta\theta^{-1}+\partial_{x}(\eta )=-u_{0}\eta\theta^{-1}+\partial_{x}(\eta ).$$ 
An easy computation shows that one has, for $\eta =u_{0}u_{1}^{a_{1}}\cdots u_{n}^{a_{n}}\omega_{0}$,
$$u_{0}\eta =\frac{1}{\mu x}F\eta -\frac{1}{\mu x}\theta (\sum_{i=1}^{n}a_{i}-w_{i})\eta.$$
Since $\theta^{2}\nabla_{\partial_{\theta}}$ is induced by the multiplication by $F$, the matrix $\nabla_{\partial_{x}}$ in the basis $\omega$ takes the form           
$$-\frac{A_{0}(x)}{\mu x\theta}+\frac{1}{\mu x}T$$
where $T$ is the diagonal matrix defined by (apply the process above to $\omega_{0}$, $\omega_{1}=u^{a(1)}\omega_{0}$, $\cdots$, $\omega_{\mu -1}=u^{a(\mu -1)}\omega_{0}$)
$$T_{kk}=\sum_{i=1}^{n}a(k-1)_{i}-w_{i}-\alpha_{k-1}.$$
Now, one has  $\sum_{i=1}^{n}a(k-1)_{i}=k-2$ (see \cite[section 2]{DoSa2}) and $\sum_{i=1}^{n}w_{i}=\mu -1$ so that $T_{kk}=k-1-\alpha_{k-1}-\mu =s_{k-1}-\mu$. Use now the symmetry property of the et $s_{k}$'s (see lemma \ref{lessk}). Of course, $R=T/\mu$.
\end{proof}

\begin{remark} 
It follows from the second part of the proposition that $(\alpha_{0},\cdots ,\alpha_{\mu -1})$ is the spectrum at infinity of any function $F_{x}:=F(. \; ,x)$, $x\in X$, see \cite[Proposition 3.2]{DoSa2}.
\end{remark}

\subsection{Towards the canonical extensions of $G$ at $x=0$}

\subsubsection{The $\varphi$-solution}
\label{solcan1}
Define $\omega_{0}^{\varphi}=\omega_{0}$ and $\omega_{1}^{\varphi}=xu_{0}\omega_{0}^{\varphi}=x\omega_{1}$ (such a choice is also natural
because $h_{\Gamma_{0}}=-\mu xu_{0}$): one has 
$$-\frac{1}{\mu}\tau\partial_{\tau}\omega_{0}^{\varphi}=\tau\omega_{1}^{\varphi}$$
and gets as above forms $\omega_{1}^{\varphi},\cdots ,\omega_{\mu -1}^{\varphi}$ satisfying      
$$-\frac{1}{\mu}(\tau\partial_{\tau}+\alpha_{k})\omega_{k}^{\varphi}=\tau\omega_{k+1}^{\varphi}$$
for all $k=0,\cdots ,\mu -2$ 
and 
$$-\frac{1}{\mu}(\tau\partial_{\tau}+\alpha_{\mu -1})\omega_{\mu -1}^{\varphi}=x\tau\omega_{0}^{\varphi}.$$
One has also              
$$\omega ^{\varphi}=\omega P$$
where $P=diag (1,x,\cdots ,x)$. 
Put, for $x\in X$, 
$$A_{0}^{\varphi}(x)=P^{-1}A_{0}(x)P$$
that is                  
$$A_{0}^{\varphi}(x)=\left ( \begin{array}{cccccc}
0   & 0   & 0 & \cdots & 0   & \mu x\\
\mu   & 0   & 0 & \cdots & 0   & 0\\
0   & \mu  & 0 & \cdots & 0   & 0\\
..  & ... & . & \cdots & .   & .\\
..  & ... & . & \cdots & .   & .\\
0   & 0   & . & \cdots & \mu   & 0 
\end{array} \right )$$
and
 $$R^{\varphi}=diag (0,0,\cdots ,0, s_{n+1}/\mu,\cdots ,s_{\mu -1}/\mu)$$
($0$ is counted $(n+1)$-times). If $\mu =n+1$, one has $R^{\varphi}=0$. $\overline{A}_{0}^{\varphi}$ will denote the value of $A_{0}^{\varphi}(x)$ at $x=0$.

\begin{proposition}
\label{basevarphi}
(1) $\omega^{\varphi}=(\omega_{0}^{\varphi},\cdots ,\omega_{\mu}^{\varphi})$ is a basis of $G_{0}$.\\
(2) In this basis, the matrix of the connection $\nabla$ takes the form 
$$(\frac{A_{0}^{\varphi}(x)}{\theta}+A_{\infty})\frac{d\theta}{\theta}+(R^{\varphi}-\frac{A_{0}^{\varphi}(x)}{\mu \theta})\frac{dx}{x}.$$
(3) For $\theta\neq 0$, the residue matrix of $\nabla_{\partial_{x}}$ at $x=0$ takes the form 
$$R^{\varphi}-\frac{\overline{A}_{0}^{\varphi}}{\mu\theta}.$$
Its eigenvalues are contained in $[0,1[$.
\end{proposition}
\begin{proof}
(1) and (2) follow from proposition \ref{libre}, using lemma \ref{lessk} and the fact that
$$R^{\varphi}\frac{dx}{x}=R\frac{dx}{x}+P^{-1}dP.$$
(3) follows because $0\leq \frac{s_{k}}{\mu}<1$. 
\end{proof}

\begin{remark}
Using the variable $\tau :=\theta^{-1}$, we find that the matrix of the connection $\nabla$ takes the form 
$$(-A_{0}^{\varphi}(x)\tau -A_{\infty})\frac{d\tau}{\tau}+(-A_{0}^{\varphi}(x)\tau -A_{\infty}+H)\frac{dx}{\mu x}$$ 
where $H$ is the diagonal matrix $diag (0,1,\cdots ,\mu -1)$. This can be used, because the entries of $H$ are integers, to show that the monodromies $T$ and $T'$ corresponding respectively to the loops around the divisors $\{\tau =0\}\times X$ and $\cit\times \{0\}$ in $\cit\times X$ are related by the formula
$$T^{-1}=(T')^{\mu}$$
(see \cite[Corollary 6.5 (ii)]{dGMS}). 
\end{remark}

\subsubsection{The $\psi$-solution}
\label{solcan2}
Define now\\

$\bullet$ $Q=diag (1,x,\cdots ,x, 1,\cdots ,1)$ ($x$ is counted $n$-times),\\       

$\bullet$ $A_{0}^{\psi}(x)=Q^{-1}A_{0}(x)Q,$\\

$\bullet$ $R^{\psi}=-diag (0,\cdots ,0,\frac{s_{\mu -1}}{\mu},\cdots ,\frac{s_{n+1}}{\mu}).$

\begin{lemma} (1) Suppose that $\mu \geq n+2$.
One has $(A_{0}^{\psi}(x))_{1,\mu}=\mu $,                  
$$(A_{0}^{\psi}(x))_{i+1,i}=\mu $$
if $i\neq n+1$ and 
$$(A_{0}^{\psi}(x))_{n+2,n+1}=\mu x.$$
(2) Suppose that $\mu =n+1$.
One has $(A_{0}^{\psi}(x))_{1,\mu}=\mu x$                 
and
$$(A_{0}^{\psi}(x))_{i+1,i}=\mu $$
if $i=1,\cdots , n$.
\end{lemma}
\begin{proof} Clear.
\end{proof}

\noindent In the sequel, $\overline{A}_{0}^{\psi}$ will denote the value of $A_{0}^{\psi}(x)$ at $x=0$. Note that $\overline{A}_{0}^{\psi}$ is regular.

\begin{proposition} (1) $\omega^{\psi}=\omega Q$ is a basis of $G_{0}$,\\
(2) In this basis, the matrix of $\nabla$ takes the form
$$(\frac{A_{0}^{\psi}(x)}{\theta}+A_{\infty})\frac{d\theta}{\theta}+(R^{\psi}-\frac{A_{0}^{\psi}(x)}{\mu\theta})\frac{dx}{x}.$$
(3) If $\theta\neq 0$, the residue matrix of $\nabla_{\partial_{x}}$ at $x=0$ takes the form
$$R^{\psi}-\frac{\overline{A}_{0}^{\psi}}{\mu\theta}.$$
Its eigenvalues are contained in $]-1,0]$.
\end{proposition}

\begin{example} If $n=2$ and $w_{1}=w_{2}=2$, one has
$$A_{0}^{\psi}(x)=5\left ( \begin{array}{ccccc}
0  & 0  & 0 & 0  & 1\\
1  & 0  & 0 & 0  & 0\\
0  & 1  & 0 & 0  & 0\\
0  & 0  & x & 0  & 0\\
0  & 0  & 0 & 1  & 0 
\end{array}
\right )$$
and
$$\overline{A}_{0}^{\psi}=5\left ( \begin{array}{ccccc}
0  & 0  & 0 & 0  & 1\\
1  & 0  & 0 & 0  & 0\\
0  & 1  & 0 & 0  & 0\\
0  & 0  & 0 & 0  & 0\\
0  & 0  & 0 & 1  & 0 
\end{array}
\right ).$$
\end{example}

\begin{remark}
\label{terminologie}
Let ${\cal L}$ (resp. ${\cal L}^{\varphi}$, ${\cal L}^{\psi}$) be the  $\cit [x, \theta ,\theta^{-1}]$-submodule of $G$ generated by $\omega$ (resp. $\omega^{\varphi}$, $\omega^{\psi}$).
One has ${\cal L}^{\varphi}\subset {\cal L}^{\psi}\subset {\cal L}$. Notice that
$$\omega_{0}=\omega_{0}^{\varphi}=\omega_{0}^{\psi}$$
and
$$R(\omega_{0})=R^{\varphi}(\omega_{0}^{\varphi})=R^{\psi}(\omega_{0}^{\psi})=0.$$
If $\mu =n+1$, one has ${\cal L}^{\varphi}= {\cal L}^{\psi}$.
\end{remark}

\begin{remark} (The flat basis) 
Let 
$\Delta$ be the open disc in $\cit$ of radius $1$, centered at $x=1$. The local basis $\omega^{flat}:=\omega x^{-R}$ ($x\in\Delta$)
of $G_{0}^{an}$ is {\em flat}: the matrix of the connection $\nabla$ takes the form 
$$(\frac{A_{0}^{flat}(x)}{\theta}+A_{\infty})\frac{d\theta}{\theta}-\frac{A_{0}^{flat}(x)}{\mu\theta}\frac{dx}{x}$$
where $A_{0}^{flat}=x^{R}A_{0}(x)x^{-R}$. Notice that the matrix $A_{0}^{flat}(x)$ has a limit $\overline{A}_{0}^{flat}$ when  $x$ approaches $0$: this is due to the fact that the sequence  $(s_{k})$ is increasing (this is equivalent to the fact that $\alpha_{k+1}\leq \alpha_{k}+1$).

\end{remark}

\section{Duality}
\label{dualite}

We define in this section a non-degenerate, symmetric and flat bilinear form on $G_{0}$.

\subsection{}
The lattice $G_{0}^{o}$ is equipped with a non-degenerate bilinear form          
$$S^{o}:G_{0}^{o}\times G_{0}^{o}\rightarrow\cit[\theta ]\theta^{n},$$
$\nabla^{o}$-flat and satisfying , for $p(\theta )\in\cit[\theta ]$,
$$p(\theta)S^{o}(\bullet \, ,\, \bullet )=S^{o}(p(\theta)\bullet \, ,\, \bullet )=S^{o}(\bullet \, ,p(-\theta)\, \bullet ).$$ 
The basis $\omega^{o}=(\omega^{o}_{0},\cdots ,\omega^{o}_{\mu -1})$ given by lemma \ref{start} is {\em adapted} to $S^{o}$: one has
$$S^{o}(\omega^{o}_{k},\omega^{o}_{j})=\delta_{k,n-j}\theta^{n}$$
for $k=0,\cdots ,n$ and
$$S^{o}(\omega^{o}_{k},\omega^{o}_{j})=\delta_{k,\mu +n-j}\theta^{n}$$
for $k=n+1,\cdots ,\mu -1$. In particular,
$A_{\infty}+A_{\infty}^{*}=nI$  and $(A_{0}^{o})^{*}=A_{0}^{o}$, where $^{*}$ denotes the adjoint with respect to $S^{o}$. 
All these results can be found in \cite[Sect. 4]{DoSa2}.\\

 We define, in the basis $\omega =(\omega_{0},\cdots ,\omega_{\mu -1})$ given by proposition \ref{libre},
$$S(\omega_{0},\omega_{n})=x^{-1}\theta^{n},$$
$$S(\omega_{k},\omega_{n-k})=x^{-2}\theta^{n}$$
for  $k=1,\cdots ,n-1$, 
$$S(\omega_{k},\omega_{\mu +n-k})=x^{-1}\theta^{n}$$
for  $k=n+1,\cdots ,\mu-1$ and
$$S(\omega_{i},\omega_{j})=0$$
otherwise. Notice that $S$ is constant in the flat basis $\omega^{flat}$: one has $S(\omega_{i}^{flat},\omega_{j}^{flat})=S^{o}(\omega_{i}^{o},\omega_{j}^{o})$
for all $i$ and for all $j$ (this follows from the symmetry property of the $s_{k}$'s).
Define now                                                                       
$$S:G_{0}\times G_{0}\rightarrow\cit[x,x^{-1},\theta ]\theta^{n}$$
by linearity, using the rules        
$$p(\theta)(\bullet \, ,\, \bullet )=S(p(\theta)\bullet \, ,\, \bullet )=S(\bullet \, ,p(-\theta)\, \bullet )$$ 
and
$$a(x)S(\bullet \, ,\, \bullet )=S(a(x)\bullet \, ,\, \bullet )=S(\bullet \, , a(x)\, \bullet )$$ 
for $p(\theta)\in\cit [\theta]$ and $a(x)\in\cit [x,x^{-1}]$. 

If $A(\theta ,x)$ denotes the matrix of the covariant derivative $\nabla_{\partial_{\theta}}$, $\check{\nabla}_{\partial_{\theta}}$ will denote the covariant derivative whose matrix is $-A(-\theta ,x)$ in the same basis.                        
We will say that $S$ is $\nabla$-{\em flat} if        
$$\partial_{\theta}S(\varepsilon , \eta )=S(\nabla_{\partial_{\theta}}(\varepsilon ) , \eta )+S(\varepsilon , \check{\nabla}_{\partial_{\theta}}(\eta) )$$
and
$$\partial_{x}S(\varepsilon , \eta )=S(\nabla_{\partial_{x}}(\varepsilon ) , \eta )+S(\varepsilon , \nabla_{\partial_{x}}(\eta) ).$$
Keep the notations of section \ref{deformation} and put $C_{0}=-A_{0}(x)/\mu x$, $C_{\infty}=R/x$, $C_{0}^{\varphi}=-A_{0}^{\varphi}(x)/\mu x$ {\em etc...}

\begin{lemma}\label{symetrie} 
(1) One has
$$(A_{0}(x))^{*}=A_{0}(x),\ A_{\infty}+A_{\infty}^{*}=nI,\ C_{0}^{*}=C_{0},$$ 
where $^{*}$ denotes the adjoint with respect to $S$. Same results for $A_{0}^{flat}(x)$, $C^{flat}$, $A_{0}^{\varphi}(x)$, $C_{0}^{\varphi}$, $A_{0}^{\psi}(x)$, $C_{0}^{\psi}$.\\ 
(2) The bilinear form $S$ is $\nabla$-flat.         
\end{lemma}
\begin{proof} 
(1) The first equality follows from the definition of $A_{0}(x)$ and from the definition of $S$. For the second, use moreover the symmetry property of the numbers $\alpha_{k}$ (see lemma \ref{lesalphak}). The third equality is then clear and (2) follows from (1) and from the definition of $S$ and $R$.
\end{proof}

\noindent We have in particular
$$x\partial_{x}S(\varepsilon , \eta )=S(R(\varepsilon ) , \eta )+S(\varepsilon , R(\eta) )$$
and this give a symmetry property for the matrix $R$.
Notice also that $S$ is the only $\nabla$-flat bilinear form which restricts to $S^{o}$ for $x=1$.
Last, notice that the coefficient of $\theta^{n}$ in $S(\varepsilon ,\eta )$, $\varepsilon ,\eta \in G_{0}$, depends only on the classes of $\varepsilon$ and  $\eta$ in $G_{0}/\theta G_{0}$. We will denote it by $g([\varepsilon ],[\eta ])$. This defines a nondegenerate bilinear form on $G_{0}/\theta G_{0}$ (see \cite[p. 211]{Sab1}).

\section{Application to the construction of Frobenius type structures and Frobenius manifolds} 
\label{construction}

\subsection{Frobenius type structures}
\label{FrTySt}

In our situation, a Frobenius type structure on $X$ is a tuple (see also \cite{Do1} and \cite{HeMa})
$$(X, E, \bigtriangledown , R_{0}, R_{\infty},\Phi ,g)$$
\noindent where\\

$\bullet$ $E$ is a  $\cit [x,x^{-1}]$-free module,\\

$\bullet$ $R_{0}$ and $R_{\infty}$ are $\cit [x,x^{-1}]$-linear endomorphisms of $E$,\\

$\bullet$ $\Phi :E\rightarrow \Omega^{1}(X)\otimes E$ is a $\cit [x,x^{-1}]$-linear map,\\

$\bullet$ $g$ is a metric on $E$, {\em i.e} a $\cit [x,x^{-1}]$-bilinear form, symmetric and nondegenerate,\\

$\bullet$ $\bigtriangledown$ is a connection on $E$,\\

\noindent these object satisfying the relations

\begin{center} $\bigtriangledown^{2}=0$, $\bigtriangledown (R_{\infty})=0$, $\Phi\wedge\Phi =0$, $[R_{0},\Phi ]=0$,
\end{center}
\begin{center} $\bigtriangledown (\Phi )=0$, $\bigtriangledown (R_{0})+\Phi =[\Phi ,R_{\infty}]$, 
\end{center} 
\begin{center}
 $\bigtriangledown (g)=0$, $\Phi^{*}=\Phi$, $R_{0}^{*}=R_{0}$, $R_{\infty}+R_{\infty}^{*}=rId$
\end{center}
for a suitable constant $r$, $^{*}$ denoting as above the adjoint with respect to $g$.\\

Keep the notations of section \ref{solnat}. 
The basis $\omega$ gives an extension of $G_{0}$ as a trivial bundle ${\cal G}$ on $\ppit^{1}\times X$ (the module of its global sections is generated by $\omega_{0},\cdots ,\omega_{\mu -1}$) equipped with a connection with logarithmic pole at $\tau :=\theta^{-1}=0$ and pole of Poincar\'e rank less or equal to one at $\theta =0$. Define 
$E:=i_{\{\theta =0\}}{\cal G}$, $E_{\infty }:=i_{\{\tau =0\}}{\cal G}$ ($E$ and $E_{\infty}$ are canonically isomorphic) and, for $i,j=0,\cdots, \mu -1$, $[\; ]$ denoting the class in $E$,\\

$\bullet$ $R_{0}[\omega_{i}]:=[\theta^{2}\nabla_{\partial_{\theta}}\omega_{i}],$\\

$\bullet$ $g([\omega_{i}],[\omega_{j}]):=\theta^{-n}S(\omega_{i},\omega_{j})$,\\

$\bullet$ $\Phi_{\xi}[\omega_{i}]:=[\theta\nabla_{\partial_{\xi}}\omega_{i}]$ for any vector field $\xi$ on $X$.\\

\noindent The connection $\bigtriangledown$ and the endomorphism $R_{\infty}$ are defined analogously, using the restriction $E_{\infty}$, $[\; ]$ denoting now the class in $E_{\infty}$,\\

$\bullet$ $R_{\infty}[\omega_{i}]:=[-\nabla_{\tau\partial_{\tau}}\omega_{i}]$,\\

$\bullet$ $\bigtriangledown_{\xi}[\omega_{i}]:=[\nabla_{\partial_{\xi}}a]$.

\begin{corollary}\label{STF}  The tuple
$(X, E, R_{0}, R_{\infty}, \Phi , \bigtriangledown , g)$ is a Frobenius type structure on $X:=\cit^{*}$. 
\end{corollary}
\begin{proof} This follows from proposition \ref{libre} and lemma \ref{symetrie} (see \cite[Chapitre V, 2]{Sab1}).
\end{proof}

\subsection{Frobenius manifolds ''in family''}
\label{constructionfrobman}

Recall that $\Delta$ denotes the open disc in $\cit$ of radius $1$, centered at $x=1$. Corollary \ref{STF} gives also an analytic Frobenius type structure 
$${\cal F}=(\Delta , E^{an}, R_{0}^{an}, R_{\infty}, \Phi^{an} , \bigtriangledown^{an} , g^{an})$$
on the {\em simply connected domain} $\Delta$. Let $\omega_{0}^{an}$ be the class of $\omega_{0}$ in $E^{an}$:  $\omega_{0}^{an}$ is $\bigtriangledown^{an}$-flat because $R(\omega_{0})=0$ (see remark \ref{terminologie}). The universal deformations and the period maps that we will consider are the ones defined in \cite{Do1} and \cite{HeMa}.

\begin{lemma} 
(1) The Frobenius type structure ${\cal F}$ has a universal deformation 
$$\tilde{{\cal F}}=(N, \tilde{E}^{an}, \tilde{R}_{0}^{an}, \tilde{R}_{\infty}, \tilde{\Phi}^{an} , \tilde{\bigtriangledown}^{an} , \tilde{g}^{an})$$
parametrized by $N:=\Delta\times (\cit^{\mu -1},0)$.\\
(2) The period map defined by the $\tilde{\bigtriangledown}^{an}$-flat extension of  $\omega_{0}^{an}$ to $\tilde{{\cal F}}$ is an isomorphism which makes $N$ a Frobenius manifold.
\end{lemma}
\begin{proof} (1) We can use the adaptation of \cite[Theorem 2.5]{HeMa} given in \cite[Section 6]{Do1} because $\omega_{0}^{an}, R_{0}^{an}(\omega_{0}^{an}),\cdots , (R_{0}^{an})^{\mu -1}(\omega_{0}^{an})$ generate $E^{an}$ and because $u_{0}:=1/u_{1}^{w_{1}}\cdots u_{n}^{w_{n}}$ is not equal to zero in $E^{an}$. (2) follows from (1) (see e.g. \cite[Theorem 4.5]{HeMa}). 
\end{proof}

\noindent The previous construction can be also done in the same way ''point by point''
(see \cite{DoSa2} and \cite{HeMa} and the references therein)
and, as quoted in the introduction, this is the classical point of view. 
Indeed, let $x\in \Delta$ and put $F_{x}:=F(.,x)$. One can attach to the Laurent polynomial  $F_{x}$ a Frobenius type structure on a point 
(see section \ref{limSTF} below) ${\cal F}_{x}^{pt}$, a universal deformation $\tilde{{\cal F}}_{x}^{pt}$ of it and finally, with the help of the section $\omega_{0}$, a Frobenius structure on $M:=(\cit^{\mu},0)$. We will call it ''the Frobenius structure attached to $F_{x}$''.
Let ${\cal F}_{x}$ ({\em resp.} $\tilde{{\cal F}}_{x}$) be the germ of ${\cal F}$ ({\em resp.} $\tilde{{\cal F}}$) at $x\in \Delta$ ({\em resp.} $(x,0)$).

\begin{proposition}
(1) The deformations  $\tilde{{\cal F}}_{x}$ and $\tilde{{\cal F}}_{x}^{pt}$ are isomorphic.\\                                   
(2) The period map defined by the flat extension of $\omega_{0}^{an}$ to $\tilde{{\cal F}}_{x}$ is an isomorphism. This yields a
Frobenius structure on $M$ which is isomorphic to the one attached to $F_{x}$.
\end{proposition}
\begin{proof}
Note first that $\tilde{{\cal F}}_{x}^{pt}$ is a deformation of ${\cal F}_{x}$: this follows from the fact that 
$u_{0}$ does not belong to the Jacobian ideal of $f$: see \cite[section 7]{Do1}. Better, $\tilde{{\cal F}}_{x}^{pt}$ is a universal deformation of 
 ${\cal F}_{x}$ because ${\cal F}_{x}$ is a deformation of ${\cal F}_{x}^{pt}$. This gives (1) because, by definition, two universal deformations of a same Frobenius type structure are isomorphic. (2) is then clear.  
\end{proof}

\noindent As a consequence, the universal deformations $\tilde{{\cal F}}_{x}^{pt}$, $x\in\Delta$, are the germs of a same section, namely $\tilde{{\cal F}}$. Thus, the Frobenius structure attached to $F_{x_{1}}$, $x_{1}\in\Delta$, can be seen as an analytic continuation of the one attached to $F_{x_{0}}$, $x_{0}\in\Delta$.

\section{Limits}

\label{limite}

Our goal is now to define a canonical limit, as $x$ approaches $0$, of the Frobenius type structure constructed in corollary \ref{STF}. We will use Deligne's canonical extensions.

\subsection{R\'esum\'e: the canonical extensions at $x=0$}

\label{extcan}
Recall the lattices defined ${\cal L}^{\varphi}$ and ${\cal L}^{\psi}$ defined in remark \ref{terminologie}. Put
$$\overline{{\cal L}}^{\varphi}:={\cal L}^{\varphi}/x{\cal L}^{\varphi}.$$
This is a free  $\cit [\theta ,\theta^{-1} ]$-module of rank $\mu$, equipped with a connection $\overline{\nabla}_{\partial_{\theta}}$ (induced by $\nabla_{\partial_{\theta}}$). 
In the sequel, $\overline{\omega}^{\varphi}$ will denote the basis of  $\overline{{\cal L}}^{\varphi}$ 
induced by $\omega^{\varphi}$. Recall also that $\overline{A}_{0}^{\varphi}$ denotes the value of $A_{0}^{\varphi}(x)$ 
at $x=0$: it is a $\mu\times\mu$ Jordan matrix. The following theorem summarizes the results obtained in the previous sections:

\begin{theorem}
\label{extdel}
 1) ${\cal L}^{\varphi}$ is equipped with a connection
$$\nabla :{\cal L}^{\varphi}\rightarrow \theta^{-1}\Omega_{\cit^{*}\times\cit}(\log (\cit^{*}\times \{0\}))\otimes {\cal L}^{\varphi}.$$
The matrix of $x\nabla_{\partial_{x}}$ in the basis $\omega^{\varphi}$ takes the form
$$-\frac{A_{0}^{\varphi}(x)}{\mu\theta}+diag(0,\cdots ,0,\frac{s_{n+1}}{\mu},\cdots ,\frac{s_{\mu -1}}{\mu})$$
and the one of $\nabla_{\partial_{\theta}}$ takes the form
$$\frac{A_{0}^{\varphi}(x)}{\theta^{2}}+\frac{A_{\infty}}{\theta}.$$
2)  $x\nabla_{\partial_{x}}$ induces a map on $\overline{{\cal L}}^{\varphi}$
whose matrix, in the basis $\overline{\omega}^{\varphi}$, takes the form
$$-\frac{\overline{A}_{0}^{\varphi}}{\mu\theta}+diag(0,\cdots ,0,\frac{s_{n+1}}{\mu},\cdots ,\frac{s_{\mu -1}}{\mu}).$$
Its eigenvalues are contained in $[0,1[$.\\
3) The matrix of $\overline{\nabla}_{\partial_{\theta}}$, acting on  $\overline{{\cal L}}^{\varphi}$, takes the form, in the basis $\overline{\omega}^{\varphi}$,
$$\frac{\overline{A}_{0}^{\varphi}}{\theta^{2}}+\frac{A_{\infty}}{\theta}.$$
\end{theorem}\qed

\begin{corollary} ${\cal L}^{\varphi}$ is Deligne's canonical extension of the bundle  $G$ to $\cit^{*}\times \cit$  such that the eigenvalues of the residue of $\nabla_{\partial_{x}}$ are contained in $[0,1[$. 
\end{corollary}

\noindent 
Analogous statements for  $\overline{{\cal L}}^{\psi}:={\cal L}^{\psi}/x{\cal L}^{\psi}$ (replace $[0,1[$ by $]-1,0]$). $\overline{{\cal L}}^{\varphi}$ is the space of the ''vanishing cycles'' and $\overline{{\cal L}}^{\psi}$ is the one of the ''nearby cycles''. More generally, and after a base change of matrix $x^{r}I$, $r\in\zit$, the lattice ${\cal L}^{\varphi}$ ({\em resp.} ${\cal L}^{\psi}$) gives 
the canonical extensions whose eigenvalues are contained in $[k,k+1[$ ({\em resp.} $]k,k+1]$).

\subsection{Limits of Frobenius type structures}

\label{limSTF}

Ideally, the limit of our Frobenius type structure as $x$ approaches $0$ should be a {\em Frobenius type structure on a point}
that is a tuple
$$(E^{lim},R_{0}^{lim}, R_{\infty}^{lim}, g^{lim})$$
where $E^{lim}$ is a finite dimensional vector space over $\cit$,  $g^{lim}$ is a symmetric and nondegenerate bilinear form on $E^{lim}$, $R_{0}^{lim}$ and $R_{\infty}^{lim}$ 
being two endomorphisms of $E^{lim}$ satisfying $(R_{0}^{lim})^{*}=R_{0}^{lim}$ and $R_{\infty}^{lim}+(R_{\infty}^{lim})^{*}=rId$ for a suitable complex number $r$, $^{*}$ denoting the adjoint with respect to $g$.
It turns out that our limit will be defined with the help of the graded module $gr^{V}(\overline{{\cal L}}^{\varphi})$ associated with the Kashiwara-Malgrange $V$-filtration at $x=0$.

\subsubsection{The $V$-filtration at $x=0$}
\label{Vfiltration} Recall the basis $\omega^{\varphi}=(\omega_{0}^{\varphi},\cdots ,\omega^{\varphi}_{\mu -1})$ of $G_{0}$ over $\cit [\theta , x, x^{-1}]$ (it is thus also a basis of $G$ over $\cit [\theta ,\theta^{-1}, x, x^{-1}]$) defined in section \ref{solcan1}.
Put $v(\omega^{\varphi}_{0})=\cdots =v(\omega^{\varphi}_{n})=0$ and, for $k=n+1,\cdots ,\mu -1$,
$v(\omega^{\varphi}_{k})=s_{k}/\mu$. Define, for $0\leq \alpha <1$,
$$V^{\alpha}G=\sum_{\alpha\leq v(\omega^{\varphi}_{k})}\cit \{x\}[\theta ,\theta^{-1}]\omega^{\varphi}_{k}+
x\sum_{\alpha >v({\omega}^{\varphi}_{k})}\cit \{x\}[\theta ,\theta^{-1}]\omega^{\varphi}_{k},$$
$$V^{>\alpha}G=\sum_{\alpha < v(\omega^{\varphi}_{k})}\cit \{x\}[\theta ,\theta^{-1}]\omega^{\varphi}_{k}+
x\sum_{\alpha \geq v(\omega^{\varphi}_{k})}\cit \{x\}[\theta ,\theta^{-1}]\omega^{\varphi}_{k}$$
and $V^{\alpha +p}G=x^{p}V^{\alpha}G$ for $p\in\zit$ and $\alpha\in [0,1[$. This gives a {\em decreasing} filtration $V^{\bullet}$ of $G$ 
by $\cit \{x\}[\theta ,\theta^{-1}]$-submodules
such that
$$V^{\alpha}G=\cit [\theta ,\theta^{-1}]<\overline{\omega}^{\varphi}_{k}|v(\omega^{\varphi}_{k})=\alpha>+V^{>\alpha}G.$$
Notice that ${\cal L}^{\varphi}=V^{0}G$ and that $\overline{{\cal L}}^{\varphi}=V^{0}G/V^{1}G$. We will put $H^{\alpha}:=V^{\alpha}G/V^{>\alpha}G$ and $H=\oplus_{\alpha\in [0,1[}H^{\alpha}$.

\begin{lemma} \label{Jordan}
(1) For each $\alpha$, $(x\nabla_{\partial_{x}}-\alpha )$ is nilpotent on $H^{\alpha}$.\\
(2) Let $N$ be the nilpotent endomorphism of $H$ which restricts to $(x\nabla_{\partial_{x}}-\alpha )$ on $H^{\alpha}$. Its Jordan blocks are in one to one correspondance with the maximal constant sequences in ${\cal S}_{w}$ and the corresponding sizes are the same.\\
(3) Let $e_{k}$ be the class of $\omega^{\varphi}_{k}$ in $H$. Then 
$e=(e_{0},\cdots ,e_{\mu -1})$ is a basis of $H$ over $\cit [\theta ,\theta^{-1}]$.
\end{lemma}
\begin{proof} (1) It suffices to prove the assertion for $\alpha \in [0,1[$. By theorem \ref{extdel}, we have
$$x\nabla_{\partial_{x}}\omega^{\varphi}_{k}=-\frac{1}{\theta}\omega^{\varphi}_{k+1}$$
for $k=0,\cdots ,n-1$ and $x\nabla_{\partial_{x}}\omega^{\varphi}_{n}\in V^{>0}G$. Moreover, for $k=n+1,\cdots ,\mu -2$ we have  
$$(x\nabla_{\partial_{x}}-\frac{s_{k}}{\mu})\omega^{\varphi}_{k}=-\frac{1}{\theta}\omega^{\varphi}_{k+1}$$
and this is equal to $0$ in $H^{v(\omega^{\varphi}_{k})}$ if  $s_{k+1}> s_{k}$. Last, 
$$(x\nabla_{\partial_{x}}-\frac{s_{\mu -1}}{\mu})\omega^{\varphi}_{\mu -1}\in x\sum_{v(\omega^{\varphi}_{\mu -1})\geq v(\omega^{\varphi}_{k})}\cit \{x\}\overline{\omega}^{\varphi}_{k}\in V^{>s_{\mu -1}/\mu}G.$$
(2) follows from (1) and (3) follows from the definition of $V^{\bullet}$.
\end{proof}

\begin{remark}
Let $B$ be the matrix of $N$ in the basis $e$: we have
$B_{i,j}=0$ if $i\neq j+1$, $B_{i+1, i}=-\frac{1}{\theta} $ if $\alpha_{i}=\alpha_{i-1}+1$ (equivalently if $s_{i}=s_{i-1}$) and $B_{i+1, i}=0$ if $\alpha_{i}\neq\alpha_{i-1}+1$. 
\end{remark}

\begin{example} \label{examplelimFTS}
(1) $n=2$ and $w_{1}=w_{2}=2$: one has ${\cal S}_{w}=(0,0,0,\frac{5}{2},\frac{5}{2})$ so that $N$ has one Jordan block of size $3$ and one Jordan block of size $2$. Its matrix in the basis $e$ is
$$-\frac{1}{\theta}\left ( \begin{array}{ccccc}
0  & 0  & 0 & 0  & 0\\
1  & 0  & 0 & 0  & 0\\
0  & 1  & 0 & 0  & 0\\
0  & 0  & 0 & 0  & 0\\
0  & 0  & 0 & 1  & 0 
\end{array}
\right ).$$
(2) $w_{1}=\cdots =w_{n}=1$: one has ${\cal S}_{w}=(0,\cdots ,0)$ so that $N$ has one Jordan block of size $\mu =n+1$. Its matrix in the basis $e$ is

$$-\frac{1}{\theta}\left ( \begin{array}{cccccc}
0   & 0   & 0 & \cdots & 0   & 0\\
1   & 0   & 0 & \cdots & 0   & 0\\
0   & 1  & 0 & \cdots & 0   & 0\\
..  & ... & . & \cdots & .   & .\\
..  & ... & . & \cdots & .   & .\\
0   & 0   & . & \cdots & 1   & 0 
\end{array}
\right ).$$

\end{example}

\begin{corollary}
The filtration $V^{\bullet}$ is the Kashiwara-Malgrange fitration at $x=0$.  

\end{corollary}
\begin{proof}
By lemma \ref{Jordan}, our filtration satisfies all the characterizing properties of the Kashiwara-Malgrange filtration at $x=0$. 
\end{proof}

\subsubsection{The canonical limit Frobenius type structure}
\label{gradue}
The module $H$ is free over $\cit [\theta ,\theta^{-1}]$ and is equipped with a connection $\nabla$ whose matrix in the basis $e$ takes the form
$$(\frac{[A_{0}]}{\theta}+[A_{\infty}])\frac{d\theta}{\theta}$$
where $[A_{0}]=-\mu\theta B$ (as above, $B$ is the matrix of $N$ in the basis $e$) and where $[A_{\infty}]$ is the diagonal matrix with eigenvalues $(\alpha_{0},\cdots ,\alpha_{\mu -1})$. 
Let $H_{0}$ be the $\cit [\theta]$-submodule of $H$ generated by $e$ and 
define 
$$k:H_{0}\times H_{0}\rightarrow \cit [\theta ]\theta^{n}$$
by
$$k(e_{k}, e_{n-k})=\theta^{n}$$
for $k=0,\cdots ,n$,
$$k(e_{k}, e_{\mu +n-k})=\theta^{n}$$
for $k=n+1,\cdots ,\mu -1$ and $k(e_{i},e_{j})=0$ otherwise. Last, put $E^{lim}=H_{0}/\theta H_{0}$. $E^{lim}$ is thus equipped with two endomorphisms $R_{0}^{lim}$ and $R_{\infty}^{lim}$ and with a non-degenerate bilinear form $g^{lim}$: if $\overline{e}_{i}$ denotes the class of $e_{i}$ in $E^{lim}$, $R_{0}^{lim}$ ({\em resp.} $R_{\infty}^{lim}$) is the endomorphism of $E^{lim}$ whose matrix is $[A_{0}]$ ({\em resp.} $[A_{\infty}]$) in the basis $\overline{e}$ and $g^{lim}$ is obtained from $k$ as in the end of section \ref{dualite}. 

\begin{theorem} \label{existlimFTS} The tuple 
 $$(E^{lim}, R_{0}^{lim}, R_{\infty}^{lim}, g^{lim})$$
is a Frobenius type structure on a point.
\end{theorem}
\begin{proof} It is enough to show that $(R_{0}^{lim})^{*}=R_{0}^{lim}$ and that $R_{\infty}^{lim}+(R_{\infty}^{lim})^{*}=nId$. 
To make the proof readable, we write $e_{i}$ instead of $\overline{e}_{i}$, $R_{0}$ instead of $R_{0}^{lim}$ {\em etc...} 
We will repeatedly use lemma \ref{lessk}, corollary \ref{lesalphak} and lemma \ref{Jordan}.\\
(a) Let us show first that $R_{\infty}+(R_{\infty})^{*}=nId$:\\
(i) assume $0\leq i\leq n$: we have $g(R_{\infty}(e_{i}),e_{j})=\alpha_{i}g(e_{i},e_{j})$ so that
$$g(R_{\infty}(e_{i}),e_{j})=\alpha_{i}$$
if $i+j=n$ and $g(R_{\infty}(e_{i}),e_{j})=0$ otherwise. In the same way,
$$g(e_{i},R_{\infty}(e_{j}))=\alpha_{j}$$
if $i+j=n$ and $g(e_{i}, R_{\infty}(e_{j}))=0$ otherwise. If $i+j=n$, one has $\alpha_{j}=\alpha_{n-i}=n-i-s_{n-i}=n-i=n-\alpha_{i}$.\\
(ii) Assume now $n+1\leq i\leq \mu -1$: we have 
$$g(R_{\infty}(e_{i}),e_{j})=\alpha_{i}$$
if $i+j=\mu +n$ and $g(R_{\infty}(e_{i}),e_{j})=0$ otherwise. In the same way,
$$g(e_{i},R_{\infty}(e_{j}))=\alpha_{j}$$
if $i+j=\mu +n$ and $g(e_{i}, R_{\infty}(e_{j}))=0$ otherwise. If $i+j=\mu +n$, one has 
$$\alpha_{j}=\alpha_{\mu  +n-i}=\mu +n-i-s_{\mu +n-i}=
\mu +n-i+(s_{i}-\mu)=n-\alpha_{i}.$$
(b) Let us show now that $(R_{0})^{*}=R_{0}$.\\
(i) Assume that $0\leq i\leq n-1$. If $i+j+1=n$ (thus $0\leq j\leq n-1$), one has
 $$g(R_{0}(e_{i}),e_{j})=g(e_{i+1},e_{j})=1$$
and
$$g(e_{i},R_{0}(e_{j}))=g(e_{i},e_{j+1})=1.$$
If $i+j+1\neq n$, one always has
 $$g(R_{0}(e_{i}),e_{j})=g(e_{i},R_{0}(e_{j}))=0.$$
(ii) Assume that $i=n$. Then $g(R_{0}(e_{n}),e_{j})=0$ because $R_{0}(e_{n})=0$ by lemma \ref{Jordan} and $g(e_{n},R_{0}(e_{j}))=0$
because $e_{0}$ does not belong to the image of $R_{0}$.\\
(iii) Assume $n+1\leq i\leq\mu -1$.\\
If $i+1+j=\mu +n$ then $s_{i+1}=s_{i}$ if and only if $s_{j+1}=s_{j}$ (because $s_{\mu +n-i}=\mu -s_{i}$). If it is the case, one has
$$g(R_{0}(e_{i}),e_{j})=g(e_{i+1},e_{j})=1$$
and
$$g(e_{i},R_{0}(e_{j}))=g(e_{i},e_{j+1})=1.$$
If $i+1+j=\mu +n$ but $s_{i+1}\neq s_{i}$ (and thus $s_{j+1}\neq s_{j}$), one has $R_{0}(e_{i})=0$ and $R_{0}(e_{j})=0$ so that
$$g(R_{0}(e_{i}),e_{j})=g(e_{i},R_{0}(e_{j}))=0.$$
If $i+1+j\neq\mu +n$, one always has
$$g(R_{0}(e_{i}),e_{j})=g(e_{i},R_{0}(e_{j}))=0.$$
\end{proof}

\begin{remark} The conclusion of the previous theorem is not always true if we do not consider the graded module $H:=gr^{V}({\cal L}^{\varphi}/x{\cal L}^{\varphi})$.
Indeed, if we work directly on $\overline{{\cal L}}^{\varphi}:={\cal L}^{\varphi}/x{\cal L}^{\varphi}$ we can define, in the same way as above, the 
tuple 
$$({\cal E}, {\cal R}_{0}, {\cal R}_{\infty}, {\cal G})$$
where\\
$\bullet$ ${\cal E}=\overline{{\cal L}}^{\varphi}/\theta \overline{{\cal L}}^{\varphi}$,\\
$\bullet$ ${\cal G}$ is the symmetric and nondegenerate bilinear form on ${\cal E}$ defined by  ${\cal G}(e_{k}', e_{n-k}')=1$ 
for $k=0,\cdots ,n$, ${\cal G}(e_{k}', e_{\mu +n-k}')=1$ for $k=n+1,\cdots , \mu -1$ and ${\cal G}(e_{i}',e_{j}')=0$ otherwise, $e_{k}'$ denoting the class of $\omega_{k}^{\varphi}$ in  ${\cal E}$,\\
$\bullet$  ${\cal R}_{0}$ (resp. ${\cal R}_{\infty}$) is the endomorphism of ${\cal E}$ whose matrix is $\overline{A}_{0}^{\varphi}$ (resp. $A_{\infty}$) in the basis $e'=(e_{0}',\cdots ,e_{\mu -1}')$.\\
\noindent The point is that this tuple  
is a Frobenius type structure on a point {\em if and only if} $\mu=n+1$: for instance, 
if $\mu\geq n+2$, we have ${\cal G}({\cal R}_{0}(e_{n}'), e_{\mu -1}')=1$ but   
${\cal G}(e_{n}', {\cal R}_{0}(e_{\mu -1}'))=0$ so that $({\cal R}_{0})^{*}\neq {\cal R}_{0}$. 
This symmetry default shows that the tuple $({\cal E}, {\cal R}_{0}, {\cal R}_{\infty}, {\cal G})$ is not a Frobenius type structure. The case $\mu =n+1$ is directly checked.
\end{remark}

\subsection{Pre-primitive sections ''at the limit''}

We will say that an element $e$ of a $\mu$-dimensional vector space $E$ over $\cit$, equipped with two endomorphisms $A$ and $B$, 
is a {\em pre-primitive section} of the triple $(E, A, B)$ if $(e,A(e),\cdots , A^{\mu -1}(e))$ is a basis of $E$ 
over $\cit$ and that $e$ is {\em homogeneous} if it is an eigenvector of $B$. Let 
$$(E^{lim}, R_{0}^{lim}, R_{\infty}^{lim}, g^{lim})$$
be the limit Frobenius structure given by theorem \ref{existlimFTS}.
Recall that $\overline{e}_{0}$ denotes the class of $\omega_{0}^{\varphi}$ in $E^{lim}$.

\begin{lemma} (1) $\overline{e}_{0}$ is a homogeneous section of the triple       
$(E^{lim}, R_{0}^{lim}, R_{\infty}^{lim})$,\\
(2) $\overline{e}_{0}$ is a pre-primitive section of the triple     
$(E^{lim}, R_{0}^{lim}, R_{\infty}^{lim})$ if and only if $\mu =n+1$. If $\mu\geq n+2$, this triple has no pre-primitive section at all.
\end{lemma}
\begin{proof} Obvious, except the last assertion: this follows from the fact that if $\mu\geq n+2$, $R_{0}^{lim}$ has at least two Jordan blocks for the same 
eigenvalue $0$ (see lemma \ref{Jordan}). 
\end{proof}

\begin{corollary}\label{limitepreprim} $\overline{e}_{0}$ is a pre-primitive and homogeneous section of the limit Frobenius type structure  
$(E^{lim}, R_{0}^{lim}, R_{\infty}^{lim}, g^{lim})$ if and only if $\mu =n+1$.
\end{corollary}

\subsection{A canonical limit Frobenius manifold}

\label{varfrob}

What do we need to construct a Frobenius manifold? In general, a Frobenius type structure and a pre-primitive and homogeneous section of it: 
the main point is that these two objects give a unique (up to isomorphism) Frobenius manifold,
see for instance \cite{HeMa} and the references to B. Dubrovin and B. Malgrange therein.

We assume here that $\mu =n+1$: theorem \ref{existlimFTS} gives a canonical limit Frobenius type structure and corollary \ref{limitepreprim} a pre-primitive and homogeneous section of it, so, as explained above, we get in this case a canonical (limit) Frobenius manifold. We can give in this case a precise description of it: 
recall that $[A_{0}]$ 
$$[A_{0}]=(n+1)\left ( \begin{array}{cccccc}
0   & 0   & 0 & \cdots & 0   & 0\\
1   & 0   & 0 & \cdots & 0   & 0\\
0   & 1   & 0 & \cdots & 0   & 0\\
..  & ... & . & \cdots & .   & .\\
..  & ... & . & \cdots & .   & .\\
0   & 0   & . & \cdots & 1   & 0 
\end{array}
\right ),\;  A_{\infty}=diag(0,1,\cdots ,n).$$
\noindent Let $x=(x_{1},\cdots ,x_{\mu})$ be a system of coordinates on $M=(\cit^{\mu},0)$.

\begin{lemma} \label{existdef} There exists a unique tuple of matrices 
$$(\tilde{A}_{0}(x), A_{\infty}, \tilde{C}_{1}(x),\cdots , \tilde{C}_{\mu}(x))$$ 
such that 
$\tilde{A}_{0}(0)=[A_{0}]$, $(\tilde{C}_{i})_{i1}=-1$ for all $i=1,\cdots ,\mu$ and satisfying the relations                                          
$$\frac{\partial \tilde{C}_{i}}{\partial x_{j}}=\frac{\partial \tilde{C}_{j}}{\partial x_{i}}$$
$$[\tilde{C}_{i}, \tilde{C}_{j}]=0$$
$$[\tilde{A}_{0}(x) , \tilde{C}_{i}]=0$$
$$\frac{\partial \tilde{A}_{0}}{\partial x_{i}}+\tilde{C}_{i}=[A_{\infty}, \tilde{C}_{i}]$$
for all  $i,j=1,\cdots ,\mu$. Precisely,\\       
(a) $\tilde{C}_{1}=-I$,\\
(b) $\tilde{C}_{2}=-J$ where $J$ denotes the nilpotent Jordan matrix of order $\mu$,\\      
(c) $\tilde{C}_{i}=-J^{i-1}$ for all $i=1,\cdots ,\mu$.\\                                      
(d) $\tilde{A}_{0}(x)=-x_{1}\tilde{C}_{1}-\mu \tilde{C}_{2}+x_{3}\tilde{C}_{3}+2x_{4}\tilde{C}_{4}+\cdots +(\mu -2)x_{\mu}\tilde{C}_{\mu}$.
\end{lemma}
\begin{proof} It is clear that the given matrices satisfy the required relations. To show unicity, note that the matrices $\tilde{C}_{i}$ are determined  by their first column  because the matrix $\tilde{A}_{0}(0)$ is regular, $\overline{e}_{0}$ is pre-primitive and the matrices $\tilde{C}_{i}$ commute with $\tilde{A}_{0}(x)$.
\end{proof}

\noindent This lemma means the following: the connection $\tilde{\nabla}$ on the bundle $\tilde{E}={\cal O}_{M}\otimes E^{lim}$ whose matrix is
$$(\frac{\tilde{A}_{0}(x)}{\theta}+A_{\infty})\frac{d\theta}{\theta}+\frac{\sum_{i=1}^{\mu}\tilde{C}_{i}dx_{i}}{\theta}$$
in the basis $\tilde{e}=(\tilde{e}_{0},\cdots , \tilde{e}_{\mu -1})=(1\otimes \overline{e}_{0},\cdots ,1\otimes \overline{e}_{\mu -1})$ of $\tilde{E}$ is {\em flat}. The matrices $\tilde{C}_{i}$ are the matrices of the covariant derivatives $\tilde{\nabla}_{\partial_{x_{i}}}$.  $\tilde{\bigtriangledown}$ will denote the connection on $\tilde{E}$ whose matrix is zero in the basis $\tilde{e}$: $\tilde{e}$ is thus the $\tilde{\bigtriangledown}$-flat extension of $\overline{e}$. We get in this way a Frobenius type structure on $M$,
$$(M,\tilde{E}, \tilde{\bigtriangledown} , \tilde{R}_{0}, \tilde{R}_{\infty}, \tilde{\Phi} , \tilde{g}).$$

\begin{corollary} \label{existFrobcan} Assume that $\mu =n+1$.\\ 
(1) The period map 
$$\varphi_{\tilde{e}_{0}} : TM\rightarrow \tilde{E},$$
$\varphi_{\tilde{e}_{0}}(\xi )=-\tilde{\Phi}_{\xi}(\tilde{e}_{0})$, is an isomorphism and $\tilde{e}_{0}$ is a homogeneous section of $\tilde{E}$, that is an eigenvector of $\tilde{R}_{\infty}$.\\
(2) The section $\tilde{e}_{0}$ defines, through the period map $\varphi_{\tilde{e}_{0}}$ a Frobenius structure on $M$ which makes $M$ a Frobenius manifold for which:\\
(a) the coordinates $(x_{1},\cdots ,x_{\mu})$ are $\bigtriangledown$-flat: one has $\bigtriangledown \partial_{x_{i}}=0$ for all  $i=1,\cdots ,\mu$,\\
(b) the product is constant in flat coordinates: $\partial_{x_{i}}*\partial_{x_{j}}=\partial_{x_{i+j-1}}$ if $i+j-1\leq \mu$, $0$ otherwise,\\
(c) the potential $\Psi$ is a polynomial of degree less or equal to $3$: $\Psi =\sum_{i,j, \; i+j\leq \mu +1}a_{ij}x_{i}x_{j}x_{\mu +2-i-j}$, up to a polynomial of degree less or equal to $2$,\\
(d) the Euler vector field is $E=x_{1}\partial_{x_{1}}+(n+1)\partial_{x_{2}}-x_{3}\partial_{x_{3}}-\cdots -(n-1)x_{n+1}\partial_{x_{n+1}},$\\
(e) the potential $\Psi$ is, up to polynomials of degree less or equal to $2$, Euler-homogeneous of degree $4-\mu$ :\\
$$E(\Psi )=(4-\mu )\Psi +G(x_{1},\cdots ,x_{\mu})$$
where $G$ is a polynomial of degree less or equal to $2$.
\end{corollary}
\begin{proof}
(1) Follows from the choice of the first columns of the matrices $\tilde{C}_{i}$: indeed, the period map $\varphi_{\tilde{e}_{0}}$ is defined by  $\varphi_{\tilde{e}_{0}}(\partial_{x_{i}})=-\tilde{C}_{i}(\tilde{e}_{0})=\tilde{e}_{i-1}$, and it is of course an isomorphism. Last, $\tilde{e}_{0}$ is homogeneous because $\overline{e}_{0}$ is so. Let us show (2): the isomorphism $\varphi_{\tilde{e}_{0}}$ brings on $TM$ the structures on $\tilde{E}$: (a) follows from the fact that the first column of the  matrices $\tilde{C}_{i}$ are constant and (b) from the fact that the matrices $\tilde{C}_{i}$ are constant because, by the definition of the product, $\varphi_{\tilde{e}_{0}}(\partial_{x_{i}}*\partial_{x_{j}})=\tilde{C}_{i}(\tilde{C}_{j}(\tilde{e}_{0}))$; (c) follows from (b) because, in flat coordinates,
$$g(\partial_{x_{i}}*\partial_{x_{j}},\partial_{x_{k}})=\frac{\partial^{3}\Psi}{\partial x_{i}\partial x_{j}\partial x_{k}}.$$
Last, (d) follows from the definition of  $\tilde{A}_{0}(x)$ and (e) is a consequence of (c) and (d).
\end{proof}
 
\noindent Of course, the period map can be an isomorphism for other choices of the first columns of the matrices $C_{i}$: whatever happens, the resulting Frobenius manifolds will be isomorphic to the one given by the corollary. Indeed, the Frobenius type structure
$$(M,\tilde{E}, \tilde{\bigtriangledown} , \tilde{R}_{0}, \tilde{R}_{\infty}, \tilde{\Phi} , \tilde{g})$$
is a {\em universal} deformation of the limit Frobenius type structure
$(E^{lim},R_{0}^{lim}, R_{\infty}^{lim}, g^{lim})$ given by theorem \ref{existlimFTS} (see \cite{HeMa}). We will thus call the Frobenius manifold given by the corollary the {\em canonical limit Frobenius manifold}.

\begin{remark}
If $\mu \geq n+2$, that is if there exists an $w_{i}$ such that $w_{i}\geq 2$, we still have a canonical limit Frobenius type structure, but no pre-primitive section of it so that the results in \cite{HeMa} do not apply. In particular, we do not know if one can find matrices as lemma \ref{existdef} (this problem is not obvious, even for the simplest examples, see for instance example \ref{example} (1)), that is if the limit Frobenius type structure and the form $\overline{e}_{0}$ (or any other) give as above a (limit) Frobenius manifold through the period map. Even if it happens to be the case, the previous construction gives then a lot of (limit) Frobenius manifolds and there is no way to compare them (we do not have any kind of unicity here). 
\end{remark}

\section{Logarithmic Frobenius type structures and logarithmic Frobenius manifolds: an example and some remarks}

Proposition \ref{libre} suggests that we are not so far from a logarithmic Frobenius type structure in the sense of \cite[Definition 1.6]{R} and one could expect at the end a logarithmic Frobenius manifold, see \cite[Definition 1.4]{R}. Some remarks are in order.

\subsection{Logarithmic Frobenius type structures}\label{LogFrTySt}

A Frobenius type structure with {\em logarithmic pole along} $\{x=0\}$ (for short, a logarithmic Frobenius type structure) is a tuple
$$(E^{log}, \{0\}, \bigtriangledown , R_{0}, R_{\infty},\Phi ,g)$$
\noindent where\\

$\bullet$ $E^{log}$ is a  $\cit [x]$-free module,\\

$\bullet$ $R_{0}$ and $R_{\infty}$ are $\cit [x]$-linear endomorphisms of $E^{log}$,\\

$\bullet$ $\Phi :E^{log}\rightarrow \Omega^{1}(\log (\{x=0\}))\otimes E^{log}$ is a $\cit [x]$-linear map,\\

$\bullet$ $g$ is a metric on $E^{log}$, {\em i.e} a $\cit [x]$-bilinear form, symmetric and non-degenerate,\\

$\bullet$ $\bigtriangledown$ is a connection on $E^{log}$ with logarithmic pole along $\{x=0\}$,\\

\noindent these object satisfying the compatibility relations of section \ref{FrTySt}. 
One can also define in an obvious way a logarithmic Frobenius type structure {\em without metric}.\\

\noindent The main point is to construct $E^{log}$: in our situation, it will be obtained from an extension of $G_{0}$ as a free $\cit [x,\theta ]$-module (and not from a canonical extension of $G$ as before). As pointed me out by C. Sevenheck, we can use for instance the $\cit [x,\theta ]$-submodule of $G_{0}$ generated by $\omega^{\varphi}_{0},\cdots ,\omega^{\varphi}_{\mu -1}$ which is a lattice ``in $x$'' of $G_{0}$. We will denote it by ${\cal L}^{\varphi}_{0}$. Let ${\cal L}^{\varphi}_{\infty}$ be the $\cit [x,\tau ]$-module
generated by $\omega^{\varphi}_{0},\cdots ,\omega^{\varphi}_{\mu -1}$ (as usual $\tau =\theta^{-1}$).

\begin{lemma}
${\cal L}^{\varphi}_{0}$ is equipped with a connection
$$\nabla :{\cal L}^{\varphi}_{0}\rightarrow \theta^{-1}\Omega_{\cit\times\Delta}(\log((\{0\}\times \cit)\cup (\cit\times \{0\})))\otimes {\cal L}^{\varphi}_{0}.$$
The matrix of $x\nabla_{\partial_{x}}$ in the basis $\omega^{\varphi}$ of ${\cal L}^{\varphi}_{0}$ takes the form
$$-\frac{A_{0}^{\varphi}(x)}{\mu\theta}+diag(0,\cdots ,0,\frac{s_{n+1}}{\mu},\cdots ,\frac{s_{\mu -1}}{\mu})$$
and the one of $\nabla_{\partial_{\theta}}$ takes the form
$$\frac{A_{0}^{\varphi}(x)}{\theta^{2}}+\frac{A_{\infty}}{\theta}.$$
\end{lemma}
\begin{proof}
Follows from proposition \ref{basevarphi}.
\end{proof}

Define $E^{log}={\cal L}^{\varphi}_{0}/\theta {\cal L}^{\varphi}_{0}$. One could imagine that the counterpart of corollary \ref{STF}: indeed, define, for $i=0,\cdots ,\mu -1$,\\

$\bullet$ $R_{0}[\omega_{i}^{\varphi}]:=[\theta^{2}\nabla_{\partial_{\theta}}\omega_{i}^{\varphi}],$\\

$\bullet$ $\Phi_{\xi}[\omega_{i}^{\varphi}]:=[\theta\nabla_{\partial_{\xi}}\omega_{i}^{\varphi}]$ for any logarithmic vector field $\xi\in Der(\log \{x=0\})$,\\

\noindent and, using the restriction of ${\cal L}^{\varphi}_{\infty}$ to $\tau =0$,\\

$\bullet$ $R_{\infty}[\omega_{i}^{\varphi}]:=[-\nabla_{\tau\partial_{\partial_{\tau}}}\omega_{i}^{\varphi}]$,\\

$\bullet$ $\bigtriangledown_{\xi}[\omega_{i}^{\varphi}]=[\nabla_{\partial_{\xi}}\omega_{i}^{\varphi}]$ for any logarithmic vector field $\xi\in Der(\log \{x=0\})$.\\

\noindent In order to define the 'metric', recall that (see section \ref{dualite})
$$S(\omega_{0}^{\varphi},\omega_{n}^{\varphi})=\theta^{n},$$
$$S(\omega_{k}^{\varphi},\omega_{n-k}^{\varphi})=\theta^{n}$$
for  $k=1,\cdots ,n-1$, 
$$S(\omega_{k}^{\varphi},\omega_{\mu +n-k}^{\varphi})=x\theta^{n}$$
for  $k=n+1,\cdots ,\mu-1$ and
$$S(\omega_{i}^{\varphi},\omega_{j}^{\varphi})=0$$
otherwise. Extend $S$ to ${\cal L}_{0}^{\varphi}$: as above, we get a flat bilinear symmetric form $g$ on $E^{log}$, 
$$g([\omega_{i}^{\varphi}], [\omega_{j}^{\varphi}]):=\theta^{-n}S(\omega_{i}^{\varphi}, \omega_{j}^{\varphi}).$$
The main point is that, of course, $g$ which is not non-degenerate, unless $\mu =n+1$.  

\begin{corollary}\label{LogSTF} (1) The tuple
$(E^{log}, \{0\}, R_{0}, R_{\infty}, \Phi , \bigtriangledown , g)$ is a logarithmic Frobenius type structure if $\mu =n+1$.\\
(2) The tuple  $(E^{log}, \{0\}, R_{0}, R_{\infty}, \Phi , \bigtriangledown)$ is a logarithmic Frobenius type structure without metric if $\mu\geq n+2$.
\end{corollary}
\begin{proof} The previous lemma finally gives a $\log (\{x=0\})-trTLEP$-structure (see \cite[Definition 1.8]{R}) if $\mu =n+1$ and we use the 1-1 correspondance between such structures and logarithmic Frobenius type structures given by \cite[Proposition 1.10]{R}.
 \end{proof}

\begin{remark} (1) We can also consider the lattice ${\cal L}_{0}$ ({\em resp.} ${\cal L}_{0}^{\psi}$), defined using the basis 
$\omega$ (resp. $\omega^{\psi}$). 
We have 
$${\cal L}_{0}^{\varphi}\subset {\cal L}_{0}^{\psi}\subset {\cal L}_{0}.$$
Anyway, we will see below that, even if we forget the metric, the 'good' one to consider is ${\cal L}^{\varphi}_{0}$.\\
(2) In \cite{dGMS}, and for $w_{1}=\cdots =w_{n}=1$, another (and more intrinsic) extension of $G_{0}$ is considered, built with the help of logarithmic vector fields.
I don't know for the moment how to compare it with ${\cal L}_{0}^{\varphi}$ and if one can define in this way extensions of $G_{0}$ if there exists an $w_{i}\geq 2$. 
\end{remark}

\subsection{Construction of a logarithmic Frobenius manifold}
\label{constructionlogfrobman}

A manifold $M$ is a {\em Frobenius manifold with logarithmic poles along the divisor} $D$ (for short a logarithmic Frobenius manifold) if $Der_{M}(\log D)$ is equipped with a metric, a multiplication and two (global)
logarithmic vector fields (the unit $e$ for the multiplication and the Euler vector field $E$), all these objects satisfying the usual compatibility relations (see \cite[Definition 1.4]{R}). We can also define a Frobenius manifold with logarithmic poles {\em without metric}: in this case, we still need a flat, torsionless connection, a symmetric Higgs field (that is a product) and two global logaritmic vector fields as before. Of course $D=\{x=0\}$ in what follows. 
According to T. Reichelt \cite[Theorem 1.12]{R} the construction in section \cite{HeMa} can be adapted to get a logarithmic Frobenius manifold from a logarithmic Frobenius type structure: let $(E^{log}, D, R_{0}, R_{\infty}, \Phi , \bigtriangledown , g)$ be a logarithmic Frobenius type structure, $\omega$ be a section of $E^{log}$. Define 
$$\varphi_{\omega}:Der_{M}(\log D)\rightarrow E^{log},$$
by
$$\varphi_{\omega}(\xi):=-\Phi_{\xi}(\omega ).$$
One says that $\omega$ satisfies\\

\noindent $\bullet$ (IC) if $\varphi_{\omega}|_{0}$ is injective,\\

\noindent $\bullet$ (GC) if $\omega |_{0}$ and its images under iteration of the maps $\Phi_{\xi}|_{0}$, $\xi\in Der_{M}(\log D)$, and $R_{0}|_{0}$ generate $E^{log}|_{0}$,\\

\noindent $\bullet$ (EC) if $\omega$ is an eigenvector of $R_{\infty}$.\\

\noindent We will say that a section of $E^{log}$ is {\em log-pre-primitive} ({\em resp. homogeneous}) if its restriction to $M-D$ is $\bigtriangledown$-flat and if it satisfies conditions (IC), (GC) ({\em resp.} (EC)). We now come back to the logarithmic Frobenius type structure given by corollary \ref{LogSTF}.

\begin{lemma} The (class of the) form $\omega_{0}^{\varphi}$ in $E^{log}$ is log-pre-primitive and homogeneous.
 \end{lemma}
\begin{proof} Follows from proposition \ref{basevarphi}: 
the flatness is given by the fact that $R^{\varphi}(\omega_{0}^{\varphi})=0$; conditions (IC) and (GC) hold because the matrix of $\Phi_{x\partial_{x}}$ is $-\frac{A_{0}^{\varphi}(x)}{\mu}$; last, we have $A_{\infty}(\omega^{\varphi}_{0})=0$ and this gives (EC). 
\end{proof}

\begin{remark} Assume that $\mu\geq n+2$: the section $\omega_{0}$ in ${\cal L}_{0}$ is flat but does not satisfy (IC) and the section $\omega_{0}^{\psi}$ in ${\cal L}_{0}^{\psi}$ is flat but does not satisfy (GC). In the former case, the only section which satisfies (IC), (EC) and (GC) is $\omega_{1}$ but this one is not flat; in the latter case, the only section which satisfies (IC), (EC) and (GC) is $\omega_{n+1}^{\psi}$ but this one is not flat. This explains why we work with ${\cal L}_{0}^{\varphi}$. 
\end{remark}

\begin{corollary}
The log-pre-primitive and homogeneous section $\omega_{0}^{\varphi}$ together with the logarithmic Frobenius type structure $(E^{log}, \{0\}, R_{0}, R_{\infty}, \Phi , \bigtriangledown , g)$ define a logarithmic Frobenius manifold if $\mu =n+1$ and a logarithmic Frobenius manifold without metric if $\mu\geq n+2$. 
\end{corollary}
\begin{proof}
Follows now from \cite[theorem 1.12]{R}. 
\end{proof}

\noindent If $\mu =n+1$, one could expect an explicit description of the logarithmic Frobenius manifold obtained, as in section \ref{varfrob}. Unfortunately, it is much more difficult and, except some trivial cases, I do not have results in this direction.

\end{document}